\documentclass[letterpaper,leqno,11pt,oneside]{amsart}

\usepackage{amsmath,amsthm,amsfonts,amssymb}
\usepackage{enumitem}
\usepackage[usenames]{color}
\usepackage{mathtools,comment}
\usepackage[extension=pdf]{hyperref}
\usepackage[height=9.6in,width=5.95in]{geometry}
\usepackage{graphics,graphicx}
\usepackage[final,color,notref,notcite]{showkeys}
\usepackage{xr}

\usepackage[style=alphabetic,natbib=true,maxnames=99,isbn=false,doi=false,url=false,firstinits=true,hyperref=auto,arxiv=abs]{biblatex}
\DeclareFieldFormat[article,inbook,incollection,inproceedings,patent,thesis,unpublished]{title}{#1\isdot}
\renewbibmacro{in:}{%
  \ifentrytype{article}{}{%
  \printtext{\bibstring{in}\intitlepunct}}}
\defbibheading{apa}[\refname]{\section*{#1}}

\definecolor{darkblue}{rgb}{0.13,0.13,0.39}
\hypersetup{colorlinks=true,urlcolor=darkblue,citecolor=darkblue,linkcolor=darkblue,%
  pdftitle={Endpoint distribution of directed polymers}}

\mathtoolsset{showonlyrefs,showmanualtags}

\newtheorem{thm}{Theorem}
\newtheorem{lem}{Lemma}[section]
\newtheorem{prop}[lem]{Proposition}

\theoremstyle{definition}
\newtheorem{rem}[lem]{Remark}
\newtheorem*{rem*}{Remark}

\newcounter{assum}

\newcommand{\I}{{\rm i}}
\newcommand{\pp}{\mathbb{P}}

\newcommand{\ee}{\mathbb{E}}
\newcommand{\rr}{\mathbb{R}}

\newcommand{\zz}{\mathbb{Z}}
\newcommand{\aip}{\mathcal{A}_2}
\newcommand{\ch}{\mathcal{H}}
\newcommand{\cb}{\mathcal{B}}
\newcommand{\ct}{\mathcal{T}}
\newcommand{\cm}{\mathcal{M}}

\newcommand{\p}{\partial}
\newcommand{\uno}[1]{\mathbf{1}_{#1}}

\newcommand{\ep}{\varepsilon}
\newcommand{\vs}{\vspace{6pt}}
\newcommand{\wt}{\widetilde}

\DeclareMathOperator{\Ai}{Ai}
\DeclareMathOperator{\tr}{tr}

\newcommand{\gref}[1]{\ref*{g-#1} of \cite{cqr}}
\newcommand{\geqref}[1]{(\ref*{g-#1}) in \cite{cqr}}
\newcommand{\grefn}[1]{\ref*{g-#1}}
\newcommand{\geqrefn}[1]{(\ref*{g-#1})}

\DeclareMathOperator*{\argmax}{arg\,max}

\numberwithin{equation}{section}

\let\oldmarginpar\marginpar
\renewcommand\marginpar[1]{\-\oldmarginpar[\raggedleft\footnotesize #1]%
{\raggedright{\small\textsf{#1}}}}

\begin{document}

\title[Endpoint distribution of directed polymers]{Endpoint distribution of directed
  polymers in \texorpdfstring{$1+1$}{1+1} dimensions}

\author{Gregorio Moreno Flores} \address[G.~Moreno~F.]{
  Department of Mathematics\\
  University of Wisconsin\\
  480 Lincoln Drive\\
  Madison, Wisconsin 53706\\
  USA } \email{moreno@math.wisc.edu} \author{Jeremy Quastel} \address[J.~Quastel]{
  Department of Mathematics\\
  University of Toronto\\
  40 St. George Street\\
  Toronto, Ontario\\
  Canada M5S 2E4} \email{quastel@math.toronto.edu} \author{Daniel Remenik}
\address[D.~Remenik]{
  Department of Mathematics\\
  University of Toronto\\
  40 St. George Street\\
  Toronto, Ontario\\
  Canada M5S 2E4 \newline \indent\textup{and}\indent
  Departamento de Ingenier\'ia Matem\'atica\\
  Universidad de Chile\\
  Av. Blanco Encala\-da 2120\\
  Santiago\\
  Chile} \email{dremenik@math.toronto.edu}

\maketitle

\begin{abstract}
  We give an explicit formula for the joint density of the max and argmax of the Airy$_2$
  process minus a parabola.  The argmax has a universal distribution which governs the
  rescaled endpoint for large time or temperature of directed polymers in $1+1$
  dimensions.
\end{abstract}

\section{Introduction}

In geometric last passage percolation, one considers a family
$\big\{w(i,j)\}_{i,j\in\zz^+}$ of independent geometric random variables with parameter
$q$ (i.e. $\pp(w(i,j)=m)=q(1-q)^{m}$ for $m\geq0$) and lets $\Pi_n$ be the collection of
up-right paths of length $n$, that is, paths $\pi=(\pi_0,\dotsc,\pi_n)$ such that
$\pi_i-\pi_{i-1}\in\{(1,0),(0,1)\}$. The \emph{point-to-point last passage time} is
defined, for $m,n\in\zz^+$, by
\[L(m,n)=\max_{\pi\in\Pi_{m+n}:(0,0)\to(m,n)}\sum_{i=0}^{m+n}w(\pi(i)),\] where the
notation in the subscript in the maximum means all up-right paths connecting the origin to
$(m,n)$. Next one defines the process $t\mapsto H_n(t)$ by linearly interpolating the
values given by scaling $L(n,y)$ through the relation
\[L(n+y,n-y)=c_1n+c_2n^{1/3}H_n(c_3n^{-2/3}y),\] where the constants $c_i$ have explicit
expressions which depend only on $q$ and can be found in \cite{johansson}. The random
variables
\[\ct_n=\inf\big\{t\!:\sup_{s\leq t}H_n(s)=\sup_{s\in\rr}H_n(s)\big\}\]
then correspond to the location of the endpoint of the maximizing path with unconstrained
endpoint. \citet{johansson} showed that
\begin{equation}
  H_n(t) \to \aip(t)-t^2
\end{equation}
in distribution, in the topology of uniform convergence on compact sets, where $\aip$ is
the Airy${}_2$ process, which is a universal limiting spatial fluctuation process in such
models, and is defined through determinantal formulas for its finite-dimensional
distributions (see the companion paper \cite{cqr} for a description). Together with known
results for last passage percolation \cite{baikRains}, Johansson's result (see also
\cite{cqr}) implies that
\begin{equation}
  \pp\!\left(\sup_{t\in\rr}\{\aip(t)-t^2\}\leq m\right)=F_{\rm GOE}(4^{1/3}m),\label{eq:johGOE}
\end{equation}
where $F_{\rm GOE}$ is the Tracy-Widom largest eigenvalue distribution for the Gaussian
Orthogonal Ensemble (GOE) from random matrix theory \cite{tracyWidom2}.

Now let $\ct$ denote
the location at which the maximum is attained,
\[\ct = \argmax_{t\in \mathbb{R}} \{\aip(t)-t^2\}.\]
Together with the recent result of \citet{corwinHammond} that the supremum of
$\aip(t)-t^2$ is attained at a unique point, Theorem 1.6 of \cite{johansson} shows

\begin{thm}
  As $n\to \infty$, $\ct_n\to\ct$ in distribution.
\end{thm}

In this article we complete the picture by providing an {\it explicit formula for the
  distribution of $\ct$}. Let $\cm$ denote the maximum of the Airy${}_2$ process minus a
parabola
\begin{equation}
  \cm = \max_{t\in \mathbb{R}} \{\aip(t)-t^2\}.
\end{equation}
Our main result is in fact an explicit formula \eqref{eq:eff} for the joint density $f(t,m)$ of
$\ct$ and $\cm$.

In the derivation of the formula, we will assume the result of \citet{corwinHammond} that
the maximum of $\aip(t)-t^2$ is obtained at a unique point.  However, we point out that it
is {\it not} necessary to do this.  In fact, if one follows the argument without this
assumption, one ends up with a formula for what is in principle a super-probability
density, i.e. a non-negative function $f(t,m)$ on $\mathbb{R}\times\mathbb{R}$ with
$\int_{\mathbb{R}\times\mathbb{R}}dm\,dt\, f(t,m)\ge 1$, and in fact one can see from the
argument that
\begin{equation}\label{eq:numberMax}
  \int_{\mathbb{R}\times\mathbb{R}}dm\,dt\, f(t,m) = {\rm expected~number~of ~maxima~of ~}\aip(t)-t^2.
\end{equation}
Recall that from \eqref{eq:johGOE} that the distribution of $\cm$ is given by a scaled
version of $F_{\rm GOE}$.  A non-trivial computation (see Section \ref{sec:uniq}) on the
resulting $f(t,m)$ gives
\begin{equation}
  \int_{-\infty}^\infty dt\, f(t,m) = 4^{1/3}F'_\mathrm{GOE}(4^{1/3}m).
\end{equation}
This shows that the resulting $f(t,m)$ has total integral one, which can only be true if
the maximum is unique almost surely.  Thus we provide an independent proof of the
uniqueness of the maximum of $\aip(t)-t^2$.

Now we state the formula.  Let $B_m$ be the integral operator with kernel
\begin{equation}
  \label{eq:Bc}
  B_m(x,y)=\Ai(x+y+m).
\end{equation}
Recall that \citet{ferrariSpohn} showed that $F_\mathrm{GOE}$ can be expressed as the
determinant
\begin{equation}
  F_\mathrm{GOE}(m)=\det(I-P_0B_mP_0),\label{eq:GOE}
\end{equation}
where $P_a$ denotes the projection onto the interval $[a,\infty)$ (the formula
  essentially goes back to \cite{sasamoto}). Here, and in everything that follows, the
determinant means the Fredholm determinant in the Hilbert space $L^2(\rr)$. In particular,
note that since $F_\mathrm{GOE}(m)>0$ for all $m\in\rr$, \eqref{eq:GOE} implies that
$I-P_0B_mP_0$ is invertible. We will write
\begin{equation}\label{eq:res}
  \varrho_m(x,y)=(I-P_0B_mP_0)^{-1}(x,y).
\end{equation}
Also, for $t,m\in\rr$ define the function
\begin{equation}
  \label{eq:phi}
  \psi_{t,m}(x)=2e^{xt}\left[t\Ai(x+m+t^2)+\Ai'(x+m+t^2)\right]
\end{equation}
and the kernel
\[\Psi_{t,m}(x,y)=2^{1/3}\psi_{t,m}(2^{1/3}x)\psi_{-t,m}(2^{1/3}y).\]
Finally, let
\begin{equation}
  \label{eq:gamma}
  \gamma(t,m)=2^{1/3}\int_0^\infty dx\int_0^\infty dy\,\psi_{-t,4^{-1/3}m}(2^{1/3}x)\varrho_{m}(x,y)\psi_{t,4^{-1/3}m}(2^{1/3}y).
\end{equation}
 
\begin{thm}\label{thm:fPE}
  The joint density $f(t,m)$ of $\ct$ and $\cm$ is given by
  \begin{equation}\label{eq:eff}
    \begin{split}
      f(t,m)&=\gamma(t,4^{1/3}m)F_\mathrm{GOE}(4^{1/3}m)\\
      &=\det\!\big(I-P_0B_{4^{1/3}m}P_0+P_0\Psi_{t,m}P_0\big)-F_\mathrm{GOE}(4^{1/3}m).
    \end{split}
  \end{equation}
\end{thm}

Integrating over $m$ one obtains a formula for the probability density $f_{\rm
  end}(t)$ of $\ct$. Unfortunately, it does not appear that the resulting integral
can be calculated explicitly, so the best formula one has is
\begin{equation}
  f_{\rm end}(t)= \int_{-\infty}^\infty dm\,f(t,m).
\end{equation}
One can readily check nevertheless that $f_{\rm end}(t)$ is symmetric in $t$. In
\cite{quastelRemTails} it is shown that the tails decay like $e^{-ct^3}$. Figure
\ref{fig:contour} shows a contour plot of the joint density of $\cm$ and $\ct$, while
Figure \ref{fig:density} shows a plot of the marginal $\ct$ density. The numerical
computations of Fredholm determinants used to produce these plots are based on the
numerical scheme and Matlab toolbox developed by F. Bornemann in
\cite{bornemann2,bornemann1}.

Although one only has the rigorous result in the case of geometric (or exponential) last
passage percolation, the key point is that the {\it polymer endpoint density} $f_{\rm end}(t)$ is expected
to be {\it universal} for directed polymers in random environment in $1+1$ dimensions,
 and even more broadly in the
KPZ universality class, for example in particle models such as asymmetric attractive interacting
particle systems (e.g. the asymmetric exclusion process), where second class particles
play the role of polymer paths.  And the analogous picture is expected to hold, as we now describe.

In the directed polymer models we consider a family $\big\{w(i,j)\}_{i\in\zz^+,j\in\zz}$ of independent identically
distributed random variables and the probability measure (polymer measure)
${\mathrm{\textbf{P}}}^w_{n,\beta}$ on the set $\Pi_n$ of one-dimensional nearest-neighbor
random walks of length $n$ starting at $0$ given by
\begin{equation} {\mathrm{\textbf{P}}}^w_{n,\beta}(\pi)= \frac{e^{\beta
      \sum_{i=0}^nw(i,\pi(i))}}{ \sum_{\pi\in\Pi_n} e^{\beta \sum_{i=0}^nw(i,\pi(i))}},
\end{equation}
where $\beta>0$ is the \emph{inverse temperature}. The analogue of $ \ct_n$ in
this context is $\pi(n)$, the random position of the endpoint.  The last passage
percolation case corresponds to $\beta =\infty$.  The infinite temperature case $\beta=0$
is nothing but a free random walk.  For $\beta<\infty$ the endpoint is random even given
the random environment $\big\{w(i,j)\}_{i\in\zz^+,j\in\zz}$. Still one expects in great
generality, and for any $\beta>0$, to have
\begin{equation}
  cn^{-2/3} \pi(n)\xrightarrow{\rm distr} f_{\rm end}
\end{equation}
for an appropriate $c$. The conjecture is that this holds whenever $E[w_+^5]<\infty$, and
fails otherwise due to the appearance of special large values of $w$ which attract the
polymer.  However, few results are available at finite temperature.  The first model for
which any results were obtained (for the free energy) is the continuum random polymer (see
below). There are now two other models, the semi-discrete model of O'Connell-Yor
\cite{OConnellYor}, and the log-Gamma polymer \cite{seppPolymerBoundary,cocsz}, for which results about
asymptotic fluctuations of the free energy are becoming available.

In the context of the continuum random polymer, we have continuous paths $x(s)$, $0<s<t$,
starting at $0$ at time $0$, with quenched random energy
\begin{eqnarray}
  & \mathcal{H}(x(\cdot)) = \int_0^t \{|\dot x(s)|^2 -\xi(s,x(s)) \}
  ds, &\label{hetch}
\end{eqnarray} where $\xi$ is Gaussian space-time white noise, that is,
$\langle\dot \xi(t,x), \dot \xi(s,y)\rangle = \delta(t-s)\delta(y-x)$. Through a
mollification procedure \cite{akq} one can  construct a probability measure $P^\xi_t$ on the space of 
continuous paths corresponding to the formal weights $e^{-\beta  \mathcal{H}}$.  It has 
finite dimensional distributions $P^\xi ( x(t_1) \in dx_1,\ldots, x(t_n) \in dx_n, x(t)
\in dx)$,  $0<t_1<\dots<t_n<t$, given by
\begin{equation}
  \frac{ Z(0,0, t_1,x_1)\cdots Z(t_{n-1},x_{n-1}, t_n, x_n) Z( t_n, x_n, t,x)}{\int dy\, Z(0,0,t,y)} dx_1\cdots dx_n\,dx
\end{equation}
where $Z(s,y,t,x)$ is the solution of the stochastic heat equation with multiplicative
noise
\[\partial_t Z=\beta^{-1}\partial_x^2 Z + \beta\xi Z\]
on $(s,t]$ with initial data $Z(s,y,s,x)= \delta(x-y)$. The temperature can be related to
time as $t\sim\beta^4$, so through a time rescaling we can set $\beta=1$ without loss of
generality.

In this setting the endpoint distribution is 
\begin{equation}
  P^\xi_t (x(t) \in dx) = \frac{ Z(0,0, t,x)}{\int dy\, Z(0,0,t,y)} dx.
\end{equation} 
Writing
\begin{equation}\label{eq:zed}
  Z(0,0, t,x) = \tfrac{1}{\sqrt{4\pi t}} e^{ -\tfrac{x^2}{4t} + (4t)^{1/3}A_t((4t)^{-2/3} x)+\tfrac{t}{24} }
\end{equation}
the key prediction (see Conjecture 1.5 of \cite{acq}) is that, as $t\to \infty$, the {\it
  crossover process} $A_t$ converges to the Airy${}_2$ process,
\begin{equation}
  A_t( x) \to \aip(x).
\end{equation}
This is proved in the sense of one dimensional marginals in \cite{acq,sasamSpohn}, and a
non-rigorous computation for multidimensional distributions was made in
\cite{ProlhacSpohn}.  Calling $\tilde{x} =(4 t)^{-2/3} x$ we can rewrite the exponent in
\eqref{eq:zed} as $(4t)^{1/3} \{ A_t (\tilde{x}) -\tilde{x}^2\}+\frac{t}{24}$, from which
we conclude that the endpoint of the polymer at time $t$ has approximately the
distribution $(4t)^{2/3} \ct$ for large $t$.  The partition functions of discrete directed
polymer models satisfy discrete versions of the stochastic heat equation, and analogous
results are expected to hold in that setting as well.

\begin{figure}
  \centering
  \includegraphics[width=4.5in]{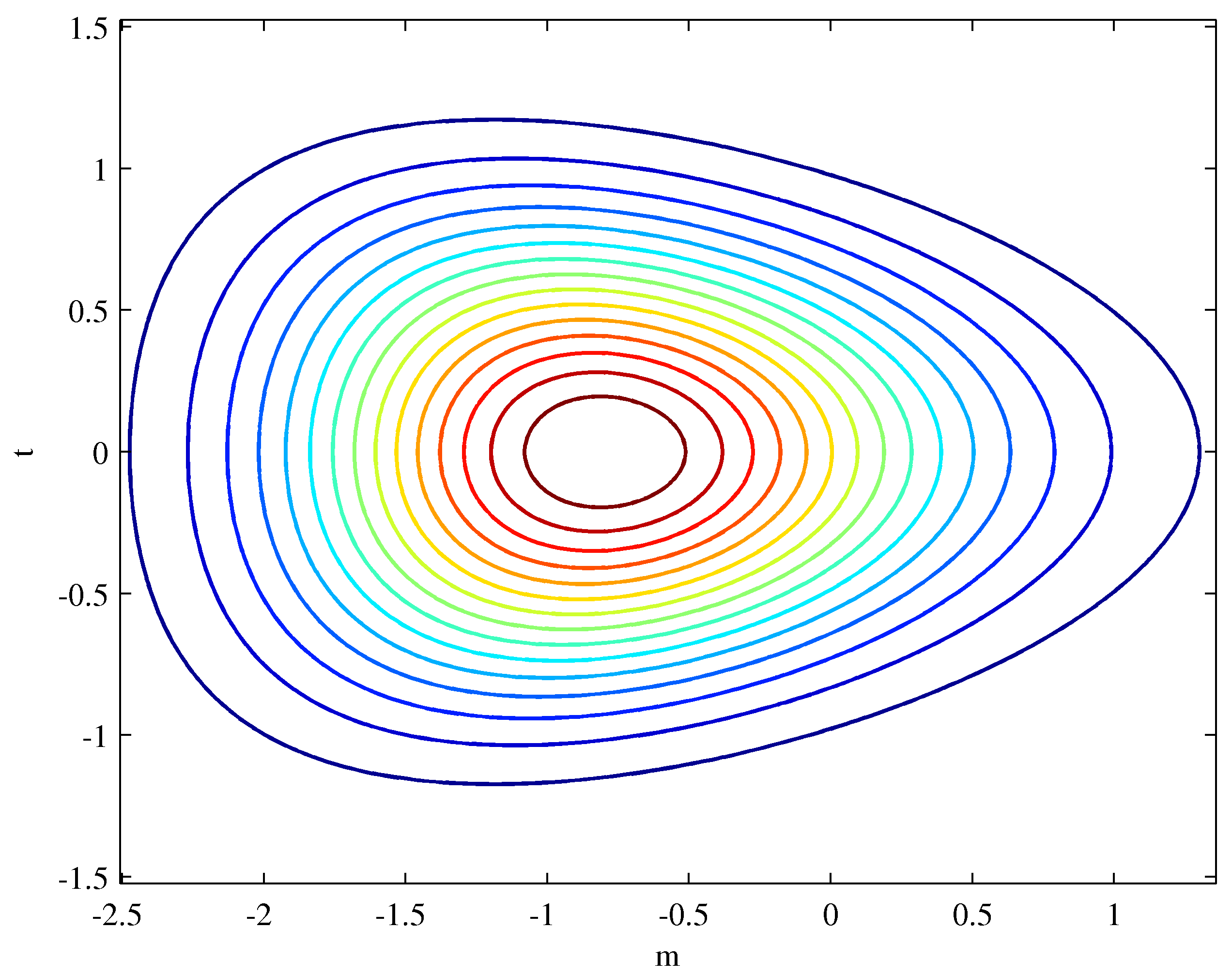}
  \vspace{-14pt}
  \caption{Contour plot of the joint density of $\cm$ and $\ct$.}
  \label{fig:contour}
\end{figure}

The problem has attracted interested in the physics literature for quite some time (see
for example \cite{mezardParisi,halpZhang}). Recently there has been a
resurgence of interest.  In particular, an alternate way to obtain the Airy${}_2$ process
is as a limit in large $N$ of the top path in a system of $N$ non-intersecting random
walks, or Brownian motions, the so called vicious walkers \cite{fisher}. \citet{SMCR},
\citet{feierl2} and \citet{RS1,RS2} obtain various expressions for the joint distributions
of $\cm$ and $\ct$ in such a system at finite $N$.  \citet{forrester} obtain the $F_{\rm
  GOE}$ distribution from large $N$ asymptotics non-rigorously, and furthermore make
connections between these problems and Yang-Mills theory. Unfortunately, the formulas
obtained for $\ct$ at finite $N$ have not been amenable to asymptotic analysis.

\begin{figure}
  \centering \hspace{-0.0in}\includegraphics[width=6in]{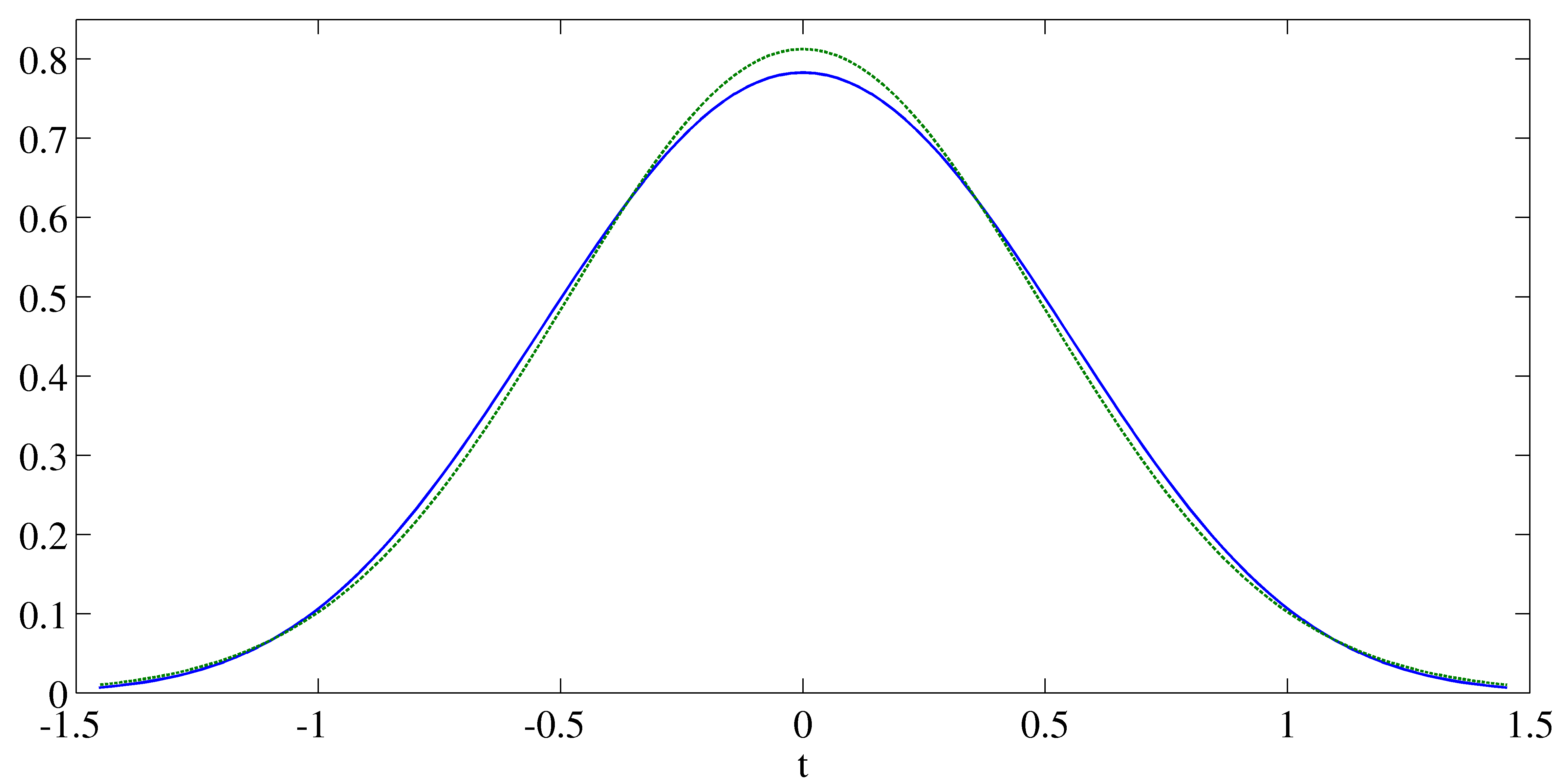} \vspace{-14pt}
  \caption{Plot of the density of $\ct$ compared with a Gaussian density with the same
    variance 0.2409 (dashed line). The excess kurtosis $\ee(\ct^4)/\ee(\ct^2)^2-3$ is
    $-0.2374$.}
  \label{fig:density}
\end{figure}

\vs

\paragraph{\bf Acknowledgements}
JQ and DR were supported by the Natural Science and Engineering Research Council of
Canada, and DR was supported by a Fields-Ontario Postdoctoral Fellowship. GMF was
supported by a postdoctoral position at the Fields Institute.  The authors thank Ivan
Corwin, Victor Dotsenko and Konstantine Khanin for interesting and helpful discussions,
and Kurt Johansson for several references to the physics literature. They also thank
Folkmar Bornemann for providing the Matlab toolbox on which numerical computations were
based. This work was done during the Fields Institute program ``Dynamics and Transport in
Disordered Systems" and the authors would like to thank the Fields Institute for its
hospitality. JQ would like in particular to acknowledge that the idea to work on this problem
arose out of discussions with Victor Dotsenko during his visit to the Fields Institute.

\section{Derivation of the formula}\label{sec:deriv}


Let $(\cm_L,\ct_L)$ denote the maximum and the location of the maximum of $\aip(t)-t^2$
restricted to $t\in[-L,L]$, and let $f_L$ be the joint density of $(\cm_L,\ct_L)$.  We
first note that, by results of I. Corwin and A. Hammond \cite{corwinHammond}, the joint
density $f(m,t)$ of $\cm, \ct$ is well approximated by $f_L(m,t)$,
\begin{equation}\label{eq:densCv}
  f(t,m) =\lim_{L\to \infty} f_L(t,m).
\end{equation}
By definition,
\[f_L(t,m) = \lim_{\delta\to0}\lim_{\ep\to0}\frac1{\ep\delta}\pp\!\left(\cm_L \in
  [m,m+\ep],\, \ct_L\in [t,t+\delta]\right),\] provided that the limit exists. The main
contribution in the above expression comes from paths entering the space-time box
$[t,t+\delta]\times[m,m+\ep]$ and staying below the level $m$ outside the time interval
$[t,t+\delta]$. More precisely, if we denote by $\underline D_{\ep,\delta}$ and $\overline
D_{\ep,\delta}$ the sets
\begin{equation}
\begin{aligned}
  \underline{D}_{\ep,\delta} &= \Big\{\aip(s)-s^2\leq m,\,\,\,
  s\in[t,t+\delta]^\text{c},\,\aip(s)-s^2\leq m+\ep,\,\,\,s\in[t,t+\delta], \\
  &\hspace{2.6in}\aip(s)-s^2 \in [m,m+\ep]\text{ for some }s\in [t,t+\delta]\Big\},\\
  \shortintertext{and} \overline{D}_{\ep,\delta}& = \Big\{\aip(s)-s^2\leq m+\ep,\,\,\,
  s\in[-L,L],\,\aip(s)-s^2 \in [m,m+\ep]\text{ for some }s\in [t,t+\delta]\Big\},
\end{aligned}
\end{equation}
then
\begin{equation}
  \underline D_{\ep,\delta}\subseteq\left\{\cm_L \in
    [m,m+\ep],\, \ct_L\in [t,t+\delta]\right\}\subseteq\overline D_{\ep,\delta}.
\end{equation}
Letting $\underline
f(t,m)=\lim_{\delta\to0}\lim_{\ep\to0}\frac1{\ep\delta}\pp\big(\underline
D_{\ep,\delta}\big)$ and defining $\overline f(t,m)$ analogously (with $\overline
D_{\ep,\delta}$ instead of $\underline D_{\ep,\delta}$) we deduce that $\underline
f(t,m)\leq f(t,m)\leq\overline f(t,m)$. In what follows we will compute $\underline
f(t,m)$. It will be clear from the argument that for $\overline f(t,m)$ we get the same
limit, so we will only compute $\underline f(t,m)$. The conclusion is that
\[f_L(t,m)=\lim_{\delta\to0}\lim_{\ep\to0}\frac1{\ep\delta}\pp\big(\underline
D_{\ep,\delta}\big).\]

We rewrite this last equation as
\begin{equation}\label{first}
  f_L(t,m)= \lim_{\delta\to0}\lim_{\ep\to0}\frac1{\ep\delta}\Big[\,\pp\!\left(\aip(s)\leq
    h_{\ep,\delta}(s),\,s\in[-L,L]\right)-\pp\!\left(\aip(s)\leq
    h_{0,\delta}(s),\,s\in[-L,L]\right)\Big],
\end{equation}
where
\[h_{\ep,\delta}(s)=s^2+m+\ep\uno{s\in[t,t+\delta]}.\] Our method is based on precise
computation of the two probabilities.  We recall the formula in Theorem \gref{thm:aiL} for
the probability that $\aip(t)\leq g(t)$ on a finite interval. Introduce the operator
$\Theta^g_{[\ell,r]}$ which acts on $L^2(\rr)$ as follows:
$\Theta^g_{[\ell,r]}f(\cdot)=u(r,\cdot)$, where $u(r,\cdot)$ is the solution at time $r$
of the boundary value problem
\begin{equation}
  \begin{aligned}
    \p_tu+Hu&=0\quad\text{for }x<g(t), ~t\in (\ell,r)\\
    u(\ell,x)&=f(x)\uno{x<g(\ell)}\\
    u(t,x)&=0\quad\text{for }x\ge g(t)
  \end{aligned}\label{eq:bdval}
\end{equation} for the \emph{Airy Hamiltonian},
\[H=-\p_x^2+x.\] In \cite{cqr} it is shown that this operator describes the height
statistics of the Airy\texorpdfstring{${}_2$}{2} process,
\begin{equation}
  \pp\!\left(\aip(t)\leq g(t)\text{ for }t\in[\ell,r]\right)
  =\det\!\left(I-K_{\Ai}+e^{-\ell H}K_{\Ai}\Theta^g_{[\ell,r]}e^{rH}K_{\Ai}\right),
  \label{eq:aiL}
\end{equation}
where we have used the cyclic property of determinants as in \geqref{eq:basiccyclic}. We
use \eqref{eq:aiL} to rewrite \eqref{first} as
\begin{align}
  f_L(t,m) &= \lim_{\delta\to0}\lim_{\ep\to0}\frac1{\ep\delta} \left[
    \det\!\left(I-K_{\Ai}+e^{LH}K_{\Ai}\Theta^{h_{\ep,\delta}}_{[-L,L]}e^{LH}K_{\Ai}\right) \right.\\
  &
  \hspace{6.5cm}\left.-\det\!\left(I-K_{\Ai}+e^{LH}K_{\Ai}\Theta^{h_{0,\delta}}_{[-L,L]}e^{LH}K_{\Ai}\right)\right].
\end{align}
The limit in $\ep$ becomes a derivative
\begin{align}
  f_L(t,m) &=\lim_{\delta\to0}\frac{1}{\delta}\,\p_\beta
  \!\left.\det\!\left(I-K_{\Ai}+e^{LH}K_{\Ai}\Theta^{h_{\beta,\delta}}_{[-L,L]}e^{LH}K_{\Ai}\right)\right|_{\beta=0},
\end{align}
which in turn gives a trace,
\begin{multline}\label{grd}
  f_L(t,m)=\det\!\left(I-K_{\Ai}+e^{LH}K_{\Ai}\Theta^{h_{0,\delta}}_{[-L,L]}e^{LH}K_{\Ai}\right)\\
  \cdot\lim_{\delta\to0}\frac{1}{\delta}
  \tr\!\left[(I-K_{\Ai}+e^{LH}K_{\Ai}\Theta^{h_{0,\delta}}_{[-L,L]}e^{LH}K_{\Ai})^{-1}
    e^{LH}K_{\Ai}\left[\p_\beta\Theta^{h_{\beta,\delta}}_{[-L,L]}\right]_{\beta=0}e^{LH}K_{\Ai}\right]
\end{multline}
(see Lemma \ref{lem:derDet} and Remark \ref{rem:lemmas}). Note that $h_{0,\delta}=g_m$,
where $g_m$ is the parabolic barrier
\begin{equation}
  \label{eq:gc}
  g_m(s)=s^2+m,
\end{equation}
so in particular the determinant and the first factor inside the trace do not depend on
$\delta$.  From \geqrefn{eq:omega} and Theorem \grefn{thm:goe} from \cite{cqr} we have
\begin{equation} \label{eq:GOEcvg} \lim_{L\to\infty}\left(
    I-K_{\Ai}+e^{LH}K_{\Ai}\Theta^{h_{0,\delta}}_{[-L,L]}e^{LH}K_{\Ai}\right)=I-A\bar
  P_0\hat R^1\bar P_0A^*
\end{equation}
in trace norm, where $\bar P_a=I-P_a$ denotes the projection onto the interval
$(-\infty,a]$,
\begin{equation}
  \hat R^1(\lambda,\tilde\lambda)=2^{-1/3}\Ai(2^{-1/3}(2m-\lambda-\tilde\lambda),\label{eq:wtR1}
\end{equation}
and the \emph{Airy transform}, $A$, acts on $f\in L^2(\rr)$ as
\[Af(x)=\int_{-\infty}^\infty dz\,\Ai(x-z)f(z).\] In particular, \geqref{eq:goe} implies
that
\begin{equation}\label{eq:detCv}
  \lim_{L\to\infty} \det\!\left(I-K_{\Ai}+e^{LH}K_{\Ai}\Theta^{h_{0,\delta}}_{[-L,L]}e^{LH}K_{\Ai}\right) = F_\mathrm{GOE}(4^{1/3}m).
\end{equation}

The next step is to compute $\p_\beta\Theta^{h_{\beta,\delta}}_{[-L,L]}\!\mid_{\beta=0}$.
Recalling that $h_{0,\delta}(s)=g_m(s)= s^2+m$ and also $h_{\ep,\delta}(s)=g_{m+\ep}(s)$
for $s\in[t,t+\delta]$ we have, by the semigroup property,
\[\Theta^{h_{\ep,\delta}}_{[-L,L]}-\Theta^{h_{0,\delta}}_{[-L,L]}
=\Theta^{g_m}_{[-L,t]}\left[\Theta^{g_{m+\ep}}_{[t,t+\delta]}
  -\Theta^{g_m}_{[t,t+\delta]}\right]\Theta^{g_m}_{[t+\delta,L]}.\] We now use Theorem
\gref{thm:thetaLgen} and a minor variation of \geqref{eq:thetaL} to obtain that
$\Theta^{g_m}_{[\ell,r]}$ has explicit integral kernel
\begin{equation}\label{eq:thetagc}
  \Theta^{g_m}_{[\ell,r]}(x,y)=\frac{e^{\ell x-ry+(r^3-\ell^3)/3}}{\sqrt{4\pi(r-\ell)}}
  \left[e^{-\frac{(x-\ell^2-y+r^2)^2}{4(r-\ell)}}-e^{-\frac{(x-\ell^2+y-r^2-2m)^2}{4(r-\ell)}}\right]
  \uno{x\leq m+\ell^2}\uno{y\leq m+r^2}.
\end{equation}
For convenience we introduce the kernels
$\vartheta_1(x,z)=e^{tz}\wt\Theta^{h_{0,0}}_{[-L,t]}(x,z)\uno{x\leq m+L^2}$ and
$\vartheta_2(\tilde z,y)=e^{-t\tilde z}\wt\Theta^{h_{0,0}}_{[t+\delta,L]}(\tilde
z,y)\uno{y\leq m+L^2}$, where $\wt\Theta^{h_{0,0}}_{[\ell,r]}$ is defined as in
\eqref{eq:thetagc} but with the indicator functions replaced by 1. Let
\begin{multline}\label{eq:defLambda}
  \Lambda^{\ep,\delta}_L(x,y)=\frac{1}{\sqrt{4\pi\delta}}e^{[(t+\delta)^3-t^3]/3}
  \int_{-\infty}^{m+t^2}dz\int_{-\infty}^{m+(t+\delta)^2}\!\!d\tilde z\,\vartheta_1(x,z)\\
  \cdot\left[ e^{-(z-t^2+\tilde z-(t+\delta)^2-2m)^2/(4\delta)}-e^{-(z-t^2+\tilde z
      -(t+\delta)^2-2m-2\ep)^2/(4\delta)} \right]\!\vartheta_2(\tilde z,y),
\end{multline}
which corresponds to $\Theta^{h_{\ep,\delta}}_{[-L,L]}-\Theta^{h_{0,\delta}}_{[-L,L]}$ but
without shifting $m$ by $\ep$ in the indicator functions in \eqref{eq:thetagc} for the
first operator in this difference. We will show in Lemma \ref{lem:deriv} that
\begin{equation}
  \label{eq:errLambda}
  \lim_{\ep\to0}\frac{1}{\ep}\left[\left(\Theta^{h_{\ep,\delta}}_{[-L,L]}-\Theta^{h_{0,\delta}}_{[-L,L]}\right)
    -\Lambda^{\ep,\delta}_L\right]=0
\end{equation}
in Hilbert-Schmidt norm. On the other hand, performing in \eqref{eq:defLambda} first the
change of variables $z\mapsto z+m+t^2$, $\tilde{z}\mapsto \tilde{z}+m+(t+\delta)^2$, then
a scaling of $z$ and $\tilde{z}$ by $\sqrt{\delta}$, and then the change of variables
$-u=z+\tilde z$, $-v=z-\tilde z$, we get
\begin{multline}
  \Lambda^{\ep,\delta}_L(x,y) =\frac{e^{[(t+\delta)^3-t^3]/3}}{4\sqrt{\pi}}\int^{\infty}_0
  du \int^u_{-u} dv\,
  \vartheta_1(x,-\sqrt\delta(u+v)/2+m+t^2)\\
  \sqrt\delta\left[ e^{-u^2/4}-e^{-(u+2\ep/\sqrt\delta )^2/4}
  \right]\vartheta_2(\sqrt\delta(v-u)/2+m+(t+\delta)^2,y),
\end{multline}
From this form and \eqref{eq:errLambda} it is straightforward to see that
\begin{multline}\label{etoo}
  \lim_{\ep\to0}
  \frac1\ep\bigg[\Theta^{h_{\ep,\delta}}_{[-L,L]}-\Theta^{h_{0,\delta}}_{[-L,L]}\bigg](x,y)
  =\lim_{\ep\to0} \frac1\ep\Lambda^{\ep,\delta}_{L}(x,y)\\
  =\frac{1}{4\sqrt{\pi}}\int^{\infty}_0 du \int^u_{-u} dv \, u \, e^{-u ^2/4}
  \,\vartheta_1(x,-\sqrt\delta(u+v)/2+m+t^2)\vartheta_2(\sqrt\delta(v-u)/2+m+(t+\delta)^2,y).
\end{multline}
The limit holds in Hilbert-Schmidt norm, as will be shown in Lemma \ref{lem:deriv}.  Now
we take the limit in $\delta$ and obtain
\begin{equation}\lim_{\delta\to
    0}\frac{1}{\delta}\left[\p_\beta\Theta^{h_{\beta,\delta}}_{[-L,L]}\right]_{\beta=0}(x,y)
  =\partial_w \vartheta_1(x,w)|_{w=m+t^2} \, \partial_w
  \vartheta_2(w,y)|_{w=m+t^2},\label{eq:grtpre}
\end{equation}
again in Hilbert-Schmidt norm, which will be checked in Lemma \ref{lem:deriv}. Referring
back to \eqref{grd} we have now shown that
\begin{equation}\label{eq:grt}
  \lim_{\delta\to
    0}\frac{1}{\delta}e^{LH}K_{\Ai}\left[\p_\beta\Theta^{h_{\beta,\delta}}_{[-L,L]}\right]_{\beta=0}
  e^{LH}K_{\Ai}=\wt\Psi_L,
\end{equation}
where $\wt\Psi_L$ has kernel
\[\wt\Psi_L(x,y)=\wt\Psi^1_L(x)\wt\Psi^2_L(y)\]
with
\begin{equation}
  \begin{aligned}
    \wt\Psi^1_L(x)&=\left.\p_w\big(e^{LH}K_{\Ai}\wt\Theta^{g_m}_{[-L,t]}M_t(x,w)\big)\right|_{w=m+t^2},\\
    \wt\Psi^2_L(y)&=\left.\p_w\big(M_{-t}\wt\Theta^{g_m}_{[t,L]}e^{LH}K_{\Ai}(w,y)\big)\right|_{w=m+t^2},
  \end{aligned}\label{eq:wtPsiL}
\end{equation}
and $M_t$ is the multiplication operator given by $M_tf(x)=e^{tx}f(x)$.

Putting \eqref{eq:densCv}, \eqref{grd}, \eqref{eq:detCv} and \eqref{eq:grt} together and
using Lemma \ref{lem:morefredholm}(a) we have
\begin{equation}
  \label{eq:fL}
  f(t,m)=\lim_{L\to\infty}\tr\!\left[(I-K_{\Ai}+e^{LH}K_{\Ai}\Theta^{g_m}_{[-L,L]}e^{LH}K_{\Ai})^{-1}\wt\Psi_L\right]F_\mathrm{GOE}(4^{1/3}m).
\end{equation}
We now have to compute the limit of the trace.  We begin by using \eqref{eq:thetagc} to
compute
\[\varphi(z):=
\p_w\big(\wt\Theta^{g_0}_{[-L,t]}M_t(z,w)\big)\Big|_{w=m+t^2}
=\frac{e^{-Lz+L^3/3+t^3/3}}{2\sqrt{\pi}(L+t)^{3/2}} (z-m-L^2)\,e^{-(z-m-L^2)^2/4(L+t)}.\]
Note how the derivative of the two terms inside the bracket in \eqref{eq:thetagc}
evaluated at $w=m+t^2$ are equal. From \eqref{eq:wtPsiL} we get
\begin{equation}
  \label{eq:dec1}
  \wt\Psi^1_L(x)=e^{LH}K_{\Ai}\bar P_{m+L^2}\,\varphi(x)
  =e^{LH}K_{\Ai}\,\varphi(x)-e^{LH}K_{\Ai}P_{m+L^2}\,\varphi(x),
\end{equation}
In Appendix \ref{sec:extras} we will show that
\begin{equation}\label{elltoo}
  \lim_{L\to\infty}\left\|e^{LH}K_{\Ai}P_{m+L^2}\,\varphi\right\|_{L^2(\rr)}=0.
\end{equation}
Now we compute $e^{LH}K_{\Ai}\varphi$. We write it as
\begin{equation}\label{eq:eLHK}
  e^{LH}K_{\Ai}\,\varphi(x) =\int_{-\infty}^0d\lambda\,e^{\lambda
    L}\Ai(x-\lambda)\int_{-\infty}^\infty dz\, \Ai(z-\lambda)\,\varphi(z).
\end{equation}
To compute the
$z$ integral, which we denote by $I(\lambda)$, we use the contour integral representation
of the Airy function given by
\begin{equation}
  \Ai(x)=\frac{1}{2\pi\I}\int_\Gamma du\,e^{u^3/3-ux},\label{eq:contour}
\end{equation}
with $\Gamma=\{c+\I s\!:s\in\rr\}$ and $c$ any positive real number, to write
\[
I(\lambda)=\frac{1}{2\pi\I}\int_\Gamma du\int_{-\infty}^\infty
dz\,e^{u^3/3-u(z+m+L^2-\lambda)}
\frac{e^{-L(z+m+L^2)+L^3/3+t^3/3}}{2\sqrt{\pi}(L+t)^{3/2}} \, e^{-z^2/4(L+t)} \, z,
\]
where we have shifted the variable $z$ by $m+L^2$. Note that the integral in $z$ is of the
form $\int_{-\infty}^\infty dz\,e^{-a_1z^2-a_2z-a_3}z$, which corresponds to computing the
mean of a certain Gaussian random variable. Performing the integration we get
\[
I(\lambda)=-\frac{2}{2\pi\I}\int_\Gamma
du\,e^{u^3/3+(L+t)u^2+(L^2+2Lt-m+\lambda)u-Lm+L^3/3+L^2 t+t^3/3}(L+u).
\]
Introducing the change of variables $u=v-L-t$ we get
\[
I(\lambda)= -\frac{2}{2\pi\I}\,e^{mt+t^3-(L+t)\lambda}\int_{\Gamma'}
dv\,e^{v^3/3-(m+t^2-\lambda)v}(v-t),
\]
where $\Gamma'$ corresponds to a shift of $\Gamma$ along the real axis. Using
\eqref{eq:contour} we deduce that
\[
I(\lambda)=2\,e^{mt+t^3-(L+t)\lambda}\Big[\!\Ai'(m+t^2-\lambda)+t\Ai(m+t^2-\lambda)\Big].
\]
Therefore
\[
e^{LH}K_{\Ai}\varphi(x) =2\int_{-\infty}^0 d\lambda\,e^{t^3+(m-\lambda)
  t}\Ai(x-\lambda)\Big[\Ai'(m+t^2-\lambda)+t\Ai(m+t^2-\lambda)\Big].
\]
We will rewrite this identity as
\[e^{LH}K_{\Ai}\varphi(x)=A\bar P_0\widetilde\psi_{t,m}(x),\] where
\begin{equation}
  \label{eq:tpsi}
  \widetilde\psi_{t,m}(x)=2e^{t^3+(m-x)t}\bigg[\Ai'(m+t^2-x)+t\Ai(m+t^2-x)\bigg].  
\end{equation}
Remarkably, the result does not depend on $L$.  Note that $A\bar
P_0\widetilde\psi_{t,m}\in L^2(\rr)$, which can be checked using the Plancherel formula
for the Airy transform
\begin{equation}
\int(Af)^2=\int\!  f^2\label{eq:planch}
\end{equation}
and the fact that
$|\!\Ai(u)|\vee|\!\Ai'(u)|\leq Ce^{-\frac23u^{3/2}}$ for some $C>0$ and all $u>0$ (see
(10.4.59-60) in \cite{abrSteg}).

Now we look at $\wt\Psi^2_L(y)$. By the time symmetry and time homogeneity of the heat
kernel it is clear that
$\p_w\big(M_{-t}\wt\Theta^{g_m}_{[t,L]}(w,\cdot)\big)(y)\big|_{w=m+t^2}$ can be obtained
from the above calculation by starting at $y$ and running backwards in time from $L$ to
$t$. Observe that the length of this time interval is $L-t$, whereas the one in the above
calculation had length $L+t$. Moreover, here we are multiplying the boundary value
operator by $M_{-t}$, whereas before we multiplied by $M_t$. It is not difficult then to
see that the answer for the second factor should be the same as for the first one, only
with $x$ replaced by $y$ and $t$ by $-t$.  From this, \eqref{eq:wtPsiL} and \eqref{elltoo}
we get that
\begin{equation}\label{eq:formwPsiL}
  \widetilde\Psi_L(x,y)\xrightarrow[L\to\infty]{}\widetilde\Psi(x,y)
  :=A\bar P_0\widetilde\psi_{t,m}(x)A\bar P_0\widetilde\psi_{-t,m}(y)
\end{equation}
in Hilbert-Schmidt sense, and thus from \eqref{eq:GOEcvg} and Lemma
\ref{lem:morefredholm}(b) we have that
\begin{equation}\label{eq:inv}
  (I-K_{\Ai}+e^{LH}K_{\Ai}\Theta^{h_{0,\delta}}_{[-L,L]}e^{LH}K_{\Ai})^{-1}\widetilde\Psi_L
  \xrightarrow[L\to\infty]{}(I-A\bar P_0\hat R^1\bar P_0A^*)^{-1}\widetilde\Psi
\end{equation}
in trace norm (the product converges in trace norm thanks to Lemma
\gref{lem:fredholm}). Therefore by Lemma \ref{lem:morefredholm}(a),
\begin{multline}
  \lim_{L\to
    \infty}\tr\!\Big[(I-K_{\Ai}+e^{LH}K_{\Ai}\Theta^{h_{0,\delta}}_{[-L,L]}e^{LH}K_{\Ai})^{-1}
  \wt\Psi_L\Big]\\
  =\tr\!\left[ (I-A\bar P_0\hat R^1\bar P_0A^*)^{-1}\widetilde\Psi\right] =\left\langle
    (I-A\bar P_0\hat R^1\bar P_0 A^*)^{-1}A\bar P_0\widetilde\psi_{t,m},A \bar
    P_0\widetilde\psi_{-t,m}\right\rangle_{L^2(\rr)},
\end{multline}
where $\langle\cdot,\cdot\rangle_\ch$ denotes inner product in the Hilbert space $\ch$
(with $\ch=L^2(\rr)$ if the subscript is omitted).
    
It only remains to simplify the expression. We use the \emph{reflection operator} $\sigma
f(x)=f(-x)$.  Because $(A\sigma)^{-1}=A\sigma$, $\sigma^2=I$ and $A^*=\sigma A\sigma$, we
have
\begin{equation}
  \left\langle (I-A\bar P_0\hat R^1\bar P_0 A^*)^{-1}A\bar P_0\widetilde\psi_{t,m},A
    \bar P_0\widetilde\psi_{-t,m}\right\rangle=\left\langle A\sigma(I-\sigma\bar
    P_0\hat R^1\bar P_0 \sigma)^{-1}\sigma\bar P_0\widetilde\psi_{t,m},A \bar
    P_0\widetilde\psi_{-t,m}\right\rangle.
\end{equation}
Since $(A\sigma)^*=A\sigma$ and $A\sigma A=\sigma$, this last term can be rewritten as
\begin{equation}
  \begin{aligned}
    \left\langle(I-\sigma\bar P_0\hat R^1\bar P_0 \sigma)^{-1}\sigma\bar
      P_0\widetilde\psi_{t,m}, (A\sigma)^*A\bar P_0\widetilde\psi_{-t,m}\right\rangle
    &=\left\langle(I-\sigma\bar P_0\hat R^1\bar P_0 \sigma)^{-1}\sigma\bar
      P_0\widetilde\psi_{t,m},\sigma\bar P_0\widetilde\psi_{-t,m}\right\rangle\\
    &\hspace{-0.4in}=\left\langle(I-P_0\sigma\hat R^1\sigma P_0)^{-1}\sigma
      \widetilde\psi_{t,m},\sigma\widetilde\psi_{-t,m}\right\rangle_{L^2([0,\infty))},
  \end{aligned}
\end{equation}
where in the second equality we used the trivial fact that $P_0\sigma=\sigma\bar P_0$ and
$\sigma P_0=\bar P_0\sigma$.  Observing that
$\sigma\widetilde\psi_{t,m}(x)=e^{t^3+mt}\psi_{t,m}(x)$, where $\psi_{t,m}$ was defined in
\eqref{eq:phi}, we deduce that
\begin{equation}
  \label{eq:trace}
  \left\langle (I-A\bar P_0\hat R^1\bar P_0 A^*)^{-1}A\bar P_0\widetilde\psi_{t,m},A
    \bar P_0\widetilde\psi_{-t,m}\right\rangle_{L^2(\rr)}=\left\langle(I-P_0\sigma\hat R^1\sigma P_0)^{-1}
    \psi_{t,m},\psi_{-t,m}\right\rangle_{L^2([0,\infty))}.
\end{equation}

Now we use the scaling operator $Sf(x)=f(2^{1/3}x)$. One can check easily that
$S^{-1}=2^{1/3}S^*$ and that $P_0$ commutes with $S$ and $S^{-1}$. Since $\sigma\hat
R^1\sigma(x,y)=2^{-1/3}\Ai(2^{-1/3}(x+y)+4^{1/3}m)$, we also have
\[S\sigma\hat R^1\sigma S^{-1}=B_{4^{1/3}m},\] where this last kernel was defined in
\eqref{eq:Bc}. Thus writing $\tilde m=2^{-1/3}m$ we get
\begin{align}
  \Big\langle(I-P_0\sigma\hat R^1\sigma
  P_0)^{-1}\psi_{t,m},\psi_{-t,m}\Big\rangle_{L^2([0,\infty))}
  &=\left\langle(I-S^{-1}P_0B_{2\tilde m}P_0S)^{-1}\psi_{t,m},\psi_{-t,m}\right\rangle_{L^2([0,\infty))}\\
  &=\left\langle
    S^{-1}(I-P_0B_{2\tilde m}P_0)^{-1}S\psi_{t,m},\psi_{-t,m}\right\rangle_{L^2([0,\infty))}\\
  &=2^{1/3}\left\langle(I-P_0B_{2\tilde
      m}P_0)^{-1}S\psi_{t,m},S\psi_{-t,m}\right\rangle_{L^2([0,\infty))},
\end{align} which is equal to $2^{1/3}\gamma(t,4^{1/3}m)$.  This gives our first formula
for $f(t,m)$ in \eqref{eq:eff}. Now observe that $\gamma(t,4^{1/3}m)$ equals the trace of
the operator $(I-P_0B_{4^{1/3}m}P_0)^{-1}P_0\Psi_{t,m}P_0$ and that $\Psi_{t,m}$ is a rank
one operator. The second equality in \eqref{eq:eff} now follows that from the general fact
that for two operators $A$ and $B$ such that $B$ is rank one, one has
$\det(I-A+B)=\det(I-A)\big[1+\tr\big((I-A)^{-1}B\big)\big]$.

\section{\texorpdfstring{$\cm$}{M} marginal and uniqueness of the maximizer}
\label{sec:uniq}

As we mentioned in the introduction, \citet{corwinHammond} showed that the maximum of
$\aip(t)-t^2$ is attained at a unique point $t\in\rr$, providing a proof of a conjecture
by K. Johansson (Conjecture 1.5 in \cite{johansson}). We used their result in Section
\ref{sec:deriv} to write formulas for $f(t,m)$ in terms of certain events concerning the
Airy$_2$ process.

Alternatively, one can turn the reasoning around and use our formula to give a different
proof of Johansson's conjecture. If we do not assume the uniqueness of the maximizer, then
the derivation in Section \ref{sec:deriv} leads to a density $f(t,m)$ for the event that
there is a maximizer at $t$ (and height $m$). Therefore the uniqueness of the maximizer is
equivalent to
\begin{equation}
  \int_{\rr^2}dt\,dm\,f(t,m)=1.\label{eq:mass}
\end{equation}
This, in turn, is a direct consequence of the following


\begin{prop}\label{prop:GOEmarginal}
  For any $m\in\rr$,
  \[\int_{-\infty}^\infty dt\,f(t,m)=\frac{d}{dm}F_{\rm GOE}(4^{1/3}m).\]
\end{prop}

\begin{proof}
  From the formula \eqref{eq:eff} for $f(t,m)$ we see that we need to compute
  \[\Psi_{m}(x,y)=\int_{-\infty}^\infty dt\,\psi_{-t,m}(\tilde x)\psi_{t,m}(\tilde y),\]
  where $\tilde x=2^{1/3}x$ and $\tilde y=2^{1/3}y$. Let
  $\Gamma_a=\{a+\I s\!:s\in\rr\}$. Then fixing $a>0$ and using \eqref{eq:contour} we have
  \[\Psi_{m}(x,y)=\frac{4}{(4\pi\I)^2}\int_{\Gamma_{2a}\times\Gamma_a}du\,dv\int_{-\infty}^\infty dt\,
  (u-t)(v+t)e^{u^3/3+v^3/3-u({\tilde x}+m+t^2)-v({\tilde y}+m+t^2)+t({\tilde x}-{\tilde
      y})}.\] The $t$ integral is just a Gaussian integral and gives
  \[\Psi_{m}(x,y)=\frac{\sqrt{\pi}}{(4\pi\I)^2}\int_{\Gamma_{2a}\times\Gamma_a}du\,dv\,
  (u+v)^{-5/2}p_{{\tilde x},{\tilde y}}(u,v)e^{q_{{\tilde x},{\tilde y}}(u,v)},\] where
  \begin{align}
    p_{{\tilde x},{\tilde y}}(u,v)&=4 u^3 v + 4 u v^3 + 8 u^2 v^2-2( u+v) + 2 (u^2 - v^2)
    ({\tilde x} - {\tilde y}) - ({\tilde x} -
    {\tilde y})^2\\
    \shortintertext{and} q_{{\tilde x},{\tilde y}}(u,v)&=\frac{\frac13(u^4 +v^4 + u^3 v +
      u v^3) - m (u + v)^2 - u^2 {\tilde x} - v^2 {\tilde y} + \frac14 ({\tilde x} -
      {\tilde y})^2 - u v ({\tilde x} + {\tilde y})}{u + v}.
  \end{align}
  Introducing the change of variables $z=u+v$, $w=u-v$, we get
  \[\Psi_{m}(x,y)=\frac{-\sqrt{\pi}}{(4\pi\I)^2}\frac12\int_{\Gamma_{3a}}dz\int_{\Gamma_a}dw\, z^{-5/2}\tilde p_{{\tilde
      x},{\tilde y}}(z,w)e^{\tilde q_{{\tilde x},{\tilde y}}(z,w)},\] where
  \begin{align}
    \tilde p_{{\tilde x},{\tilde y}}(z,w)&= - w^2 z^2 + 2wz ({\tilde x} - {\tilde
      y})-({\tilde x}-{\tilde y})^2-2 z+z^4 \shortintertext{and} \tilde q_{{\tilde
        x},{\tilde y}}(z,w)&=\frac{w^2 z^2- 2 w z({\tilde x} - {\tilde y})+ ({\tilde x} -
      {\tilde y})^2 - 2 ({\tilde x} + {\tilde y}+2m) z^2 + \frac13z^4}{4z}.
  \end{align}
  Changing variables $w\mapsto iw$, the $w$ integral is another Gaussian integral and we
  get
  \begin{align}
    \Psi_{m}(x,y)&=\frac{1}{4\pi\I}\int_{\Gamma_{3a}}dz\,z\, e^{z^3/12-z({\tilde
        x}+{\tilde y}+2m)/2} =\frac{4^{2/3}}{4\pi\I}\int_{\Gamma_{4^{-1/3}3a}}dz\,
    z\,e^{z^3/3-2^{-1/3}z({\tilde x}+{\tilde y}+2\tilde
      m)}\\
    &=-2^{1/3}\Ai'\!\big(x+y+4^{1/3}m)\big)=-2^{-1/3}\p_mB_{4^{1/3}m}(x,y),
  \end{align}
  where we have used \eqref{eq:contour}. Using this in the definition of $\gamma(t,m)$ we
  deduce that
  \[\int_{-\infty}^\infty
  dt\,\gamma(t,m)=-2^{1/3}\tr\!\left[\left(I-P_0B_mP_0\right)^{-1}\Psi_{m}\right]
  =-\tr\!\left[\left(I-P_0B_mP_0\right)^{-1}\p_mB_m\right].\] Consequently we get from
  \eqref{eq:eff} and \eqref{eq:GOE} that
  \[\int_{-\infty}^\infty
  dt\,f(t,m)=-\tr\!\left[\left(I-P_0B_mP_0\right)^{-1}\p_mB_m\right]
  \det\!\big(I-P_0B_mP_0\big) =\frac{d}{dm}\det\!\big(I-P_0B_mP_0\big),\] where the last
  inequality follows from Lemma \ref{lem:derDet}. The result now follows from
  \eqref{eq:GOE}.
\end{proof}

\appendix

\section{Technical estimates}\label{sec:extras}

Section \gref{sec:aiL} contains a short review of some general facts about trace class and
Hilbert-Schmidt operators and Fredholm determinants. In Section \ref{sec:deriv} of the
present article we used some additional facts, which we state next. Here $\ch$ will denote
a separable Hilbert space and $\cb_1(\ch)$ will denote the space of \emph{trace class
  operators} in $\ch$, which is endowed with the trace norm (see Section \gref{sec:aiL}
for a short discussion or \cite{simon} for a complete treatment).

\begin{lem}\label{lem:morefredholm}
  \mbox{} Assume $\big\{\!A(v)\big\}_{v\geq0}$ is a family of operators converging as
  $v\to\infty$ in $\cb_1(\ch)$ to some operator $A\in\cb_1(\ch)$. Then:
  \begin{enumerate}[label=(\alph*),itemsep=4pt]
  \item $\displaystyle \tr(A(v))\xrightarrow[v\to\infty]{}\tr(A)$.
  \item If $I-A(v)$ is invertible for all large enough $v$ and $I-A$ is also invertible,
    then
    \[(I-A(v))^{-1}\xrightarrow[v\to\infty]{}(I-A)^{-1}\quad\text{in }\cb_1(\ch).\]
  \end{enumerate}
\end{lem}

This result comes from Theorem 3.1 and Corollary 5.2 in \cite{simon}. Using (5.1) from
\cite{simon} one can also easily show the following (see also the corollary just cited):

\begin{lem}\label{lem:derDet}
  Assume $\big\{\!A(\beta)\big\}_{\beta\in[0,1)}$ is a family of operators in $\cb_1(\ch)$
  such that there is an operator $\p_\beta A(0)$ satisfying
  \[\frac1\beta\left[A(\beta)-A(0)\right]\xrightarrow[\beta\to0]{}\p_\beta
  A(0)\quad\text{in }\cb_1(\ch).\] Then the map
  $\beta\longmapsto\det\!\big(I+A(\beta)\big)$ is differentiable at 0 and
  \[\p_\beta\det\!\big(I+A(\beta)\big)\Big|_{\beta=0}
  =\tr\!\left[(I+A(0))^{-1}\p_\beta A(0)\right]\det\!\big(I+A(0)\big).\]
\end{lem}

\begin{rem}\label{rem:lemmas}
  Note that the last two lemmas assume convergence in trace norm as the hypothesis.
  Throughout Section \ref{sec:deriv} (see \eqref{grd}, \eqref{etoo} and \eqref{eq:grtpre})
  we used these results for operators of the form $e^{LH}K_{\Ai}\Phi_\eta e^{LH}K_{\Ai}$,
  where $\eta$ is some parameter and we know that $\Phi_\eta$ converges in Hilbert-Schmidt
  norm to some limit $\Phi$. As we will see in Lemma \ref{lem:deriv}, the convergence is
  in fact a bit stronger, and using this we can justify the application of the lemmas in
  Section \ref{sec:deriv}. To see why, note that if we let $\varphi(x)=1+x^2$ and define
  the multiplication operator $Mf(x)=\varphi(x)f(x)$ then by Lemma \gref{lem:fredholm} we
  have
  \begin{equation}
    \|e^{LH}K_{\Ai}(\Phi_\eta-\Phi)e^{LH}K_{\Ai}\|_1
    \leq\|e^{LH}K_{\Ai}\|_{\rm op}\|(\Phi_\eta-\Phi)M\|_2\|M^{-1}e^{LH}K_{\Ai}\|_2.\label{eq:remlem}
  \end{equation}
  By \eqref{eq:eLHK} we have, for $f\in L^2(\rr)$,
  \begin{align}
      \|e^{LH/2}K_{\Ai}f\|_2^2&
      =\int_{\rr^3}dx\,dy\,d\tilde y\int_{(-\infty,0]^2}d\lambda\,d\tilde\lambda\,
      e^{(\lambda+\tilde\lambda)L/2}\Ai(x-\lambda)\Ai(y-\lambda)f(y)\\
      &\hspace{3in}\cdot\Ai(x-\tilde\lambda)
      \Ai(\tilde y-\tilde\lambda)f(\tilde y)\\
      &=\int_{\rr^2}dy\,d\tilde y\int_{(-\infty,0]^2}d\lambda\,d\tilde\lambda
      \,e^{(\lambda+\tilde\lambda)L/2}\Ai(y-\lambda)f(y)
      \Ai(\tilde y-\tilde\lambda)f(\tilde y)\delta_{\lambda=\tilde\lambda}\\
      &=\int_{-\infty}^0d\lambda\,e^{\lambda L}Af(\lambda)^2.
  \end{align}
  Using \eqref{eq:planch} we deduce that
  $\|A\|_{\rm op}=\|A^*\|_{\rm op}=1$,  and then
  \begin{equation}\label{eq:eLHKnorm}
     \|e^{LH/2}K_{\Ai}\|_{\rm op}\leq1.
  \end{equation}
  The third norm in \eqref{eq:remlem} is also finite, thanks to \geqref{eq:sndHS}, and we
  are going to prove below the convergence $\|(\Phi_\eta-\Phi)M\|_2\to0$ in each relevant
  case.\noeqref{eq:eLHKnorm}
\end{rem}

The next result provides the missing estimates in the proof of \eqref{eq:fL}.

\begin{lem}\label{lem:deriv}
  For each fixed $\delta,L>0$, the convergences in \eqref{eq:errLambda}, \eqref{etoo} and
  \eqref{eq:grtpre} hold in Hilbert-Schmidt norm. Moreover, if we let $\varphi(x)=1+x^2$
  and define the multiplication operator $Mf(x)=\varphi(x)f(x)$, then the three
  convergences above still hold if we multiply each operator on the right by $M$.
\end{lem}

\begin{proof}
  The second equality in (\ref{etoo}) one follows from the dominated convergence theorem
  and the estimate
  \begin{equation}
    \left| \frac{\sqrt\delta}{\ep}\left[ e^{-u^2/4}-e^{-(u+2\ep/\sqrt\delta )^2/4} \right]
      - u e^{-u ^2/4} \right| \leq C \frac{\ep}{\sqrt\delta}(1+u^2)e^{-u^2/4},\label{eq:bdDer}
  \end{equation}
  where $C>0$ can be taken uniform in $u\geq0$ for small enough $\ep$. Using this bound
  and the particular form of $\vartheta_1$ and $\vartheta_2$ we can see that
  \begin{eqnarray*}
    &&\bigg| \int^{\infty}_0 du \int^u_{-u} dv \, \left\lbrace \frac{\sqrt\delta}{\ep}\left[ e^{-u^2/4}-e^{-(u-2\ep/\sqrt\delta )^2/4} \right]
      - u e^{-u ^2/4} \right\rbrace   \\
    && \hspace{2.5cm} \cdot \,  \vartheta_1(x,-\sqrt\delta(u+v)/2+c+t^2) \, \vartheta_2(\sqrt\delta(v-u)/2+m+(t+\delta)^2,y) \bigg| \\
    && \hspace{1cm}\leq C
    \frac{\ep}{\sqrt\delta}e^{C(|x|+|y|)-\frac{x^2+y^2}{C}},
  \end{eqnarray*} 
  for some $C>0$. Integrating the square of the left side with respect to $x$ and $y$ over
  $(-\infty,m+L^2]^2$, we can deduce again by the dominated convergence theorem that
  $\ep^{-1}\Lambda^{\ep,\delta}_L$ converges in Hilbert-Schmidt norm. This, together with
  \eqref{eq:errLambda}, proves (\ref{etoo}).

  Next we observe that
  \begin{multline}
    \Big|\vartheta_1(x,-\sqrt\delta(u+v)/2+m+t^2)\vartheta_2(\sqrt\delta y(v-u)/2+m+(t+\delta)^2,y) \\
    +\frac{\delta}{4}(u+v)(v-u) \partial_w\vartheta_1(x,w)|_{w=m+t^2}\,\partial_w\vartheta_2(w,y)|_{w=m+t^2}\Big|
    \leq \delta^{3/2} \, e(u,v,x,y),
  \end{multline}
  where $e$ involves products of first and second derivatives of $\vartheta_1$ and
  $\vartheta_2$. By the same argument we explained above, $\int du \int dv \,
  |e(u,v,x,y)|$ can be easily seen to be in $L^2((-\infty,m+L^2]^2)$ as a function of $x$
  and $y$. Thus, by the dominated convergence theorem,
  \begin{multline}
    \lim_{\delta\to0}\frac{1}{\delta}\left[\p_\beta\Theta^{h_{\beta,\delta}}_{[-L,L]}\right]_{\beta=0}(x,y) \\
    =-\frac{1}{4\sqrt{\pi}} \int^{\infty}_0 du \int^u_{-u} dv \, u \,(u+v)(v-u) \, e^{-u
      ^2/4} \, \frac14 \partial_w \vartheta_1(x,w)|_{w=m+t^2} \, \partial_w
    \vartheta_2(w,y)|_{w=m+t^2}
  \end{multline}
  in $L^2((-\infty,m+L^2]^2)$. The integral in $u$ and $v$ can be computed, and gives the
  answer $-16\sqrt{\pi}$, so we deduce that
  \[\frac{1}{\delta}\left[\p_\beta\Theta^{h_{\beta,\delta}}_{[-L,L]}\right]_{\beta=0}(x,y)
  \xrightarrow[\delta\to0]{}\partial_w \vartheta_1(x,w)|_{w=m+t^2} \, \partial_w
  \vartheta_2(w,y)|_{w=m+t^2}\] in Hilbert-Schmidt norm. This proves \eqref{eq:grtpre}.

  We are left with proving \eqref{eq:errLambda}. Let
  $E_\ep=\big(\Theta^{h_{\ep,\delta}}_{[-L,L]}-\Theta^{h_{0,\delta}}_{[-L,L]}\big)
  -\Lambda^{\ep,\delta}_L$. To simplify notation we assume $m=t=0$, for the general case
  the proof is exactly the same. From \eqref{eq:thetagc} and \eqref{eq:defLambda} we have
  \begin{equation}
    E_\ep(x,y)=\frac{1}{\sqrt{4\pi\delta}}e^{\delta^3/3}\int_Ddz\,d\tilde z\,\vartheta_1(x,z)\!
    \left[e^{-(z-\tilde z+\delta^2)^2/(4\delta)}-e^{-(z+\tilde z-\delta^2-2\ep)^2/(4\delta)}\right]\!\vartheta_2(\tilde z,y),
  \end{equation}
  where
  $D=\big((-\infty,\ep]\times(-\infty,\ep+\delta^2]\big)\setminus\big((-\infty,0]\times(-\infty,\delta^2]\big)$. We
  split $D$ into the union of three disjoint regions of pairs $(z,\tilde z)$: $D_1=\{0\leq
  z\leq\ep,\delta^2\leq\tilde z\leq\delta^2+\ep\}$, $D_2=\{0\leq z\leq\ep,\tilde z<0\}$
  and $D_3=\{z<0,\delta^2\leq\tilde z\leq\delta^2+\ep\}$. Similarly we split $E_\ep$ as
  the sum of the integrals $E^i_\ep$ over each region. On the first region we have
  \[\frac{1}{\ep}\left|E^1_\ep(x,y)\right|\leq\frac{2}{\sqrt{4\pi\delta}}e^{\delta^3/3}\frac{1}{\ep}
  \int_{D_1}dz\,d\tilde z\,\vartheta_1(x,z) \vartheta_2(\tilde
  z,y)\xrightarrow[\ep\to0]{}0\] thanks to the particular form of $\vartheta_1$ and
  $\vartheta_2$ and the fact that $D_1$ has area
  $\ep^2$.  
  For the second region we have
  \[\lim_{\ep\to0}\frac{1}{\ep}E^2_\ep(x,y)=\frac{e^{\delta^3/3}}{\sqrt{4\pi\delta}}
  \int_{-\infty}^0d\tilde z\!\left.\vartheta_1(x,z)\!  \left[e^{-(z-\tilde
        z+\delta^2)^2/(4\delta)}-e^{-(z+\tilde
        z-\delta^2)^2/(4\delta)}\right]\!\vartheta_2(\tilde z,y)\right|_{z=0}=0,\] while
  the third region can be dealt with analogously. We deduce by the triangle inequality
  that $\ep^{-1}|E_\ep(x,y)|\to0$ as $\ep\to0$.  To upgrade the convergence to
  Hilbert-Schmidt norm we may use the dominated convergence theorem and similar estimates
  as for (b) and (c), we omit the details. This finishes the proof of (a).

  Finally, it is straightforward to check in each case that the convergences still hold if
  we multiply each kernel by the polynomial $1+y^2$.
\end{proof}

\begin{proof}[Proof of \eqref{elltoo}]
  By Lemma \grefn{lem:fredholm} and \eqref{eq:eLHKnorm} we have
  \[\|e^{LH}K_{\Ai}P_{m+L^2}\varphi\|_2\leq\|e^{LH}K_{\Ai}\|_{\rm
    op}\|P_{m+L^2}\varphi\|_2\leq\|P_{m+L^2}\varphi\|_2.\]
  This last norm can be easily computed:
  \begin{align}
    \|P_{m+L^2}\varphi\|_2^2
    &= \frac{1}{16 \pi (L+t)^3} e^{\frac23 L^3 + \frac23 t^3} \int^{\infty}_{m+L^2} dz\, (z-m+L^2)^2e^{-\frac{(z-m-L^2)^2}{2(L+t)}-2Lz}\\
    &= \frac{1}{16 \pi (L+t)^3} e^{-\frac43 L^3 + \frac23 t^3} \int^{\infty}_{m} dz\,
    (z-m)^2e^{-\frac{(z-m)^2}{2(L+t)}-2Lz}.
  \end{align}
  Let $F_L(z)$ denote the argument in the last exponential. $F_L$ is minimized at
  $z^*=m-2L(L+t)$, which is less than $m$ for large $L$, and is strictly increasing in
  $[z^*,\infty)$. Thus $F_L$ attains its minimum inside the interval $[m,\infty)$ at
  $z=m$, where its value is $-2mL$. An application of Laplace's method (Lemma
  \gref{lem:laplace}) then shows that
  \[\|P_{m+L^2}\varphi\|_2^2 \leq Ce^{-L^ 3/C},\] for some $C>0$, which finishes the
  proof.
\end{proof}

\printbibliography[heading=apa]

\end{document}